\documentclass[reqno]{elsarticle}
\usepackage{xcolor}
\usepackage[utf8]{inputenc}
\usepackage{hyperref}
\usepackage{enumitem}
\usepackage{soul}
\usepackage{cancel}

\long\def\MSC#1\EndMSC{\def\arg{#1}\ifx\arg\empty\relax\else {\narrower\noindent%
{2010 Mathematics Subject Classification}: #1\\} \fi}
\long\def\PACS#1\EndPACS{\def\arg{#1}\ifx\arg\empty\relax\else
     {\narrower\noindent%
{PACS numbers}: #1}\fi}
\long\def\KEY#1\EndKEY{\def\arg{#1}\ifx\arg\empty\relax\else
	{\narrower\noindent%
Keywords: #1\\}\fi}

\usepackage{a4wide,amsmath}
\usepackage{amssymb}
\usepackage{mathtools}
\usepackage{mathrsfs}
\usepackage{amsthm}
\usepackage{mathabx}
\numberwithin{equation}{section}

\theoremstyle{plain}
\newtheorem{theorem}{Theorem}[section]
\newtheorem{lemma}[theorem]{Lemma}

\theoremstyle{definition}
\newtheorem{definition}[theorem]{Definition}

\theoremstyle{remark}
\newtheorem{remark}[theorem]{Remark}
\setenumerate{label=(\roman*)}
\usepackage{tikz}
\usetikzlibrary[calc, math, decorations.pathreplacing]
\newcommand{\N}{\mathbb{N}}	
\newcommand{\Z}{\mathbb{Z}}
\newcommand{\R}{\mathbb{R}}
\newcommand{\C}{\mathbb{C}}

\newcommand{\J}{\mathbf{j}}
\newcommand{\K}{\mathbf{k}}

\newcommand{\T}{\mathcal{Q}}

\def\bb1{{\rm{1}\hspace{-3pt}\mathbf{l}}}

\renewcommand{\d}{\;\mathrm{d}} 

\renewcommand{\epsilon}{\varepsilon}
\renewcommand{\phi}{\varphi}

\begin{document}

\title{Sharp iteration asymptotics for transfer operators induced by greedy $\beta$-expansions}


\author[add1]{Horia D. Cornean}
\ead{cornean@math.aau.dk}

\author[add1]{Kasper S. S{\o}rensen\corref{cor1}}
\ead{kasper@math.aau.dk}

\cortext[cor1]{Corresponding author}

\address[add1]{Department of Mathematical Sciences, Aalborg University, Thomas Manns Vej 23, 9220 Aalborg, Denmark}

\begin{abstract}
We consider base-$\beta$ expansions of Parry's type, where $a_0 \geq a_1 \geq 1$ are integers and $a_0<\beta <a_0+1$ is the positive solution to $\beta^2 = a_0\beta + a_1$ (the golden ratio corresponds to $a_0=a_1=1$). The map $x\mapsto \beta x-\lfloor \beta x\rfloor$ induces a discrete dynamical system on the interval $[0,1)$ and we study its associated transfer (Perron-Frobenius) operator $\mathscr{P}$. Our main result can be roughly summarized as follows: we explicitly construct two piecewise affine functions $u$ and $v$ with $\mathscr{P}u=u$ and $\mathscr{P}v=\beta^{-1} v$ such that for every sufficiently smooth $F$ which is supported in $[0,1]$ and satisfies $\int_0^1 F \d x=1$, we have $\mathscr{P}^kF= u +\beta^{-k}\big ( F(1)-F(0)\big )v +o(\beta^{-k})$ in $L^\infty$. This is also compared with the case of integer bases, where more refined asymptotic formulas are possible.
\end{abstract}


\maketitle

\tableofcontents

\section{Introduction and main results}
For any real $\beta>1$ we may define a discrete dynamical system on the interval $[0,1)$ induced by the map 
\begin{equation}\label{hc00} 
T_\beta: [0,1)\to [0,1),\quad T_\beta(x):=\beta x-\lfloor \beta x\rfloor.
\end{equation}
Every such system induces for $x\in [0,1)$ a ``greedy'' $\beta$-expansion of the type
\begin{equation}\label{hc000}
x=\sum_{k\geq 1} x_k \beta^{-k},\quad x_k:=\lfloor \beta T_\beta^{k-1}(x)\rfloor. 
\end{equation}
The ergodic properties of such a dynamical system are strongly related to the properties of the so-called composition (Koopman) operator given by 
\[
(K_\beta f)(x)= f(T_\beta(x)),\quad f\in L^\infty([0,1]),  
\]
and its associated transfer (Perron-Frobenius) operator \cite{Suz,Wal,LY, CHM}  $P_\beta: L^1([0,1])\to L^1([0,1])$ which is defined by ``duality'':
\begin{equation}\label{ks1}
\int_0^1(P_\beta f)(x) g(x)\d x= \int_0^1 f(x) (K_\beta g)(x)\d x,\quad f\in L^1([0,1]), \, g\in L^\infty([0,1]). 
\end{equation}
In particular, $T_\beta$ has a unique absolutely continuous invariant probability measure \cite{Bl,BG,DK,DKr,Go} with density $u_\beta\geq 0$ if $u_\beta$ is the unique solution to the eigenvalue problem 
\begin{equation}\label{eq:hc107}
P_\beta \, u =u,\quad u \geq 0 \,\,  \text{a.e.},\quad \int_0^1 u (x)  \d x=1.
\end{equation}
Since $P_\beta$ is a Markov operator \cite{Br} it follows that $\Vert P_\beta\Vert_{L^1\to L^1}=1$. Using this, one possible way of proving the existence and uniqueness of a solution to \eqref{eq:hc107} is to investigate whether the sequence $\{P_\beta^kf\}_{k\geq 1}$ converges in $L^1$ when $k\to\infty$, and the limit is independent of the initial condition $f$, provided $\int_0^1 f \d x=1$ and $f$ belongs to a dense subset of $L^1$. A more refined convergence result (see \cite{CHM,HMS,HMS2} for certain classes of $\beta$'s) is to show that if the initial condition $f$ has a certain regularity, for example if it belongs to $C^1([0,1])$, then there exist two constants $C\geq 0$ and $1/2 \leq c<1$ such that for every $f\in C^1([0,1])$ one has
\[
\Big \Vert P_\beta^k f- u_\beta \int_0^1 f\d t \Big \Vert_{L^1}\leq \Vert f\Vert_{C^1([0,1])}\, C\, \beta ^{-c\, k}.
\]
Related results regarding the speed of convergence when the Perron-Frobenius operator is restricted to sufficiently smooth functions, were obtained in \cite[Theorem 2.5]{Li} and in \cite{BKL,Mo1,Mo2,PS}.

The main goal of the current paper is to identify certain classes of $\beta$'s and $f$'s for which it is possible to derive a two-term asymptotic expansion, where the error is proportional to $\beta^{-ck}$ with $c>1$.

\subsection{The integer bases are special}
Let us assume that $\beta\equiv q\geq 2$ is a natural number. In this case \cite{Ga, HMS}:
\begin{align*}
   (P_q f)(x)\equiv (\T f)(x)=\frac{1}{q}\sum_{j=0}^{q-1}f\bigg (\frac{x+j}{q}\bigg ).
\end{align*}
Let us define ${B}_0(x)=1$ and ${B}_1(x)=x-1/2$ for $x\in [0,1)$ and $\Z$-periodically extend both $B_0$ and $B_1$ to all of $\mathbb{R}$; the periodic functions are denoted by $\widetilde{B}_0$ and $\widetilde{B}_1$. Let $n\geq 2$ and define 
\begin{align*}
\widetilde{B}_n(x)\coloneqq -n!\sum_{m\neq 0} e^{2\pi i m x}\, (2\pi i m)^{-n},\quad \frac{d\widetilde{B}_n}{dx}(x)= n\widetilde{B}_{n-1}(x),\quad 0<x<1.
\end{align*}
Then $\widetilde{B}_n$ are $\mathbb{Z}$-periodic functions which coincide with the Bernoulli polynomials $B_n$ on $(0,1)$ for $n \in \mathbb{N}$. If $f\in C^N([0,1])$ with $N\geq 1$, the Euler-Bernoulli approximation formula \cite{AS} reads as:
\begin{equation}\label{june17}
    \begin{aligned}
     f(x)=   \int_0^1 f(t) \d t &+\sum_{s=1}^N \big ( f^{(s-1)}(1)-f^{(s-1)}(0)\big ) \frac{\widetilde{B}_s(x)}{s!}  \\
     &-\int_{0}^{1}\, f^{(N)}(u)\, \frac{\widetilde{B}_N(x-u)}{N!} \d u,\quad 0< x< 1.  
    \end{aligned}
\end{equation}
For completeness, we will give a short proof of this formula in \ref{sec:eulerbernoulliformula}. In \ref{lem:lemA2} We will also show that (the case $q=2$ has been shown in \cite{Dr})
\[\T \widetilde{B}_j = q^{-j} \widetilde{B}_j,\quad j\geq 0.\]
This means that the Bernoulli polynomials are joint eigenfunctions for all possible transfer operators when $\beta\equiv q\geq 2$ is an integer.  Introducing this in \eqref{june17} leads to
\begin{equation}\label{hc2}
\begin{aligned}
  (\T^k f)(x)&=\int_0^1 f(t) \d t +\sum_{j=1}^N \big (q^{-j}\big )^k \Big (f^{(j-1)}(1)-f^{(j-1)}(0)\Big )\frac{\widetilde{B}_j(x)}{j!}\\ 
  &\qquad - q^{- N  k}\int_0^1 f^{(N)}(t)\, \frac{\widetilde{B}_N(x-t)}{N!}  \d t,\quad \forall \, k\geq 1.
  \end{aligned} 
\end{equation}
This not only proves that the density $u_\beta$ of the invariant absolutely continuous probability measure always equals $\chi_{[0,1]}$ for all integers $q\geq 2$, i.e. the invariant measure itself always equals Lebesgue measure (this is a well-known fact, see e.g. \cite[Example 6.3.2]{DK} or \cite[Proposition 2.2]{CHMSS}), but we also get a complete asymptotic expansion with terms proportional with $q^{-k}$, $q^{-2k}$ and so on, depending on how smooth $f$ is. We will see that such a strong result is generally not possible when $\beta$ is not an integer.

\subsection{The main result}
As we have already mentioned, our main interest is to investigate whether more refined asymptotic formulas like \eqref{hc2} exist for $\beta$-expansions \cite{Pa,Re} when the base is not an integer.
In the current paper we will only focus on the special case where $\beta>1$ is the unique positive solution to
\begin{align}\label{eq:fundamentaleqmoregeneral}
    1 = \frac{a_0}{\beta} + \frac{a_1}{\beta^2},\quad a_0\geq a_1\geq 1,
\end{align}
with $a_0,a_1 \in \mathbb{N}.$
In this simple case $\beta$ can be explicitly computed:
\begin{align}\label{eq:a0betaa0plus1}
 \beta= \frac{a_0+\sqrt{a_0^2+4a_1}}{2},\quad    a_0 < \beta < a_0+1\, .
\end{align}
Furthermore, the transfer operator $P_\beta \equiv \mathscr{P}$, defined on $L^1$ functions supported on $[0,1]$, is given by:
\begin{align}\label{ks5}
    (\mathscr{P} f)(x) = \frac{1}{\beta}\sum_{j=0}^{a_0} f\bigg( \frac{x+j}{\beta}\bigg).
\end{align}
This is a special case of a more general formula for piecewise $C^2$ monotone interval maps, see \cite[(6.8)]{DK}. We include the proof in Lemma~\ref{lemmaA1} for completeness.
Let us introduce the following four piecewise affine functions:
\begin{equation}
\begin{aligned}\label{ks6}
    & \psi_1(x) \coloneqq \chi_{[0,1]}(x), \quad 
    \psi_2(x) \coloneqq \frac{\beta}{a_1}\chi_{[0,a_1\beta^{-1}]}(x),\\
    &\psi_3(x) \coloneqq  4\bigg(x-\frac{1}{2}\bigg) \chi_{[0,1]}(x), \quad 
    \psi_4(x) \coloneqq  \frac{4\beta^2}{a_1^2} \bigg(x-\frac{a_1}{2\beta}\bigg) \chi_{[0,a_1\beta^{-1}]}(x),
\end{aligned}
\end{equation}
where $\Vert \psi_j\Vert_{L^1}=1,\, j\in\{1,2,3,4\}$ (see Figure~\ref{fig:psionetofour}). We will need three linear combinations of them:
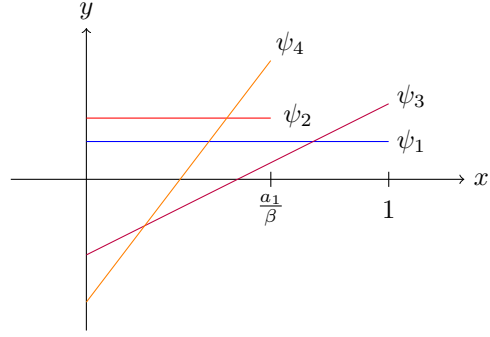
\begin{figure}[!h]
\begin{center}
\begin{tikzpicture}
  \draw[->] (-1, 0) -- (5, 0) node[right] {$x$};
  \draw[->] (0, -2) -- (0, 2) node[above] {$y$};
  \draw[xscale=4, yscale=0.5, domain=0:1, smooth, variable=\x, blue] plot ({\x}, {1});
  \draw[xscale=4, yscale=0.5,domain=0:0.61, smooth, variable=\y, red]  plot ({\y}, 1.62);
  \draw[xscale=4, yscale=0.5,domain=0:1, smooth, variable=\x, purple] plot ({\x}, {4*(\x-1/2)});
  \draw[xscale=4, yscale=0.5,domain=0:0.61, smooth, variable=\y, orange]  plot ({\y}, {4*2.62*(\y-0.31)});
  \draw[-] (4, -0.1) -- (4, 0.1) node[below] {};
  \draw[-] (2.44, -0.1) -- (2.44, 0.1) node[below] {};
  \node[] at (2.44, -0.4)   (a) {$\frac{a_1}{\beta}$};
  \node[] at (4, -0.4)   (a) {$1$};
   \node[] at (4.3, 0.5)   (a) {$\psi_1$};
\node[] at (2.8, 0.85)   (a) {$\psi_2$};
  \node[] at (4.3, 1.1)   (a) {$\psi_3$};
  \node[] at (2.7, 1.8)   (a) {$\psi_4$};
\end{tikzpicture}
\caption{Plot of $\psi_1$ (blue), $\psi_2$ (red), $\psi_3$ (purple) and $\psi_4$ (orange).}
\label{fig:psionetofour}
\end{center}
\end{figure}
 \begin{equation}\label{hc20}
    \begin{aligned}
        &\tilde{u}_1= \frac{\beta^2}{\beta^2+a_1} \, \psi_1 +\frac{a_1}{\beta^2+a_1} \, \psi_2 ,\\ 
        & \tilde{u}_2= \frac{a_1}{\beta^2+a_1} \, \psi_1 -\frac{a_1}{\beta^2+a_1} \, \psi_2 ,\\
        &\tilde{u}_3= -\frac{2a_0a_1\beta}{(\beta+a_1)(\beta^2+a_1)} \psi_1 +\frac{2a_0a_1\beta}{(\beta+a_1)(\beta^2+a_1)}\psi_2+\frac{\beta^2}{(\beta^2+a_1)} \psi_3+\frac{a_1^2\beta^{-1}}{(\beta^2+a_1)}\psi_4.
    \end{aligned}
\end{equation}

\begin{theorem}\label{thm1} 
     Let $\mathscr{P}$ be the transfer operator in \eqref{ks5}. Then the functions $\tilde{u}_1$, $\tilde{u}_2$ and $\tilde{u}_3$ in \eqref{hc20} are eigenfunctions for $\mathscr{P}$ in $L^\infty([0,1])$ and  
    \begin{equation}\label{hor5}
    \begin{aligned}
    &\mathscr{P}{\tilde{u}}_1={\tilde{u}}_1,\quad \mathscr{P}{\tilde{u}}_2=\Big (-\frac{a_1}{\beta^2}\Big ){\tilde{u}}_2,\quad \mathscr{P}{\tilde{u}}_3=\beta^{-1}{\tilde{u}}_3,\\
    &\int_0^1 \tilde{u}_1\d x=1,\quad \int_0^1 \tilde{u}_2\d x=\int_0^1 \tilde{u}_3\d x=0.
    \end{aligned}
    \end{equation}
    Moreover, for every fixed $N\in \mathbb{N}$ which satisfies 
    \begin{align*}
        N> \frac{3\ln \big (\frac{\beta^2}{a_1}\big )}{\ln\big(\frac{\beta}{a_1}\big)},
    \end{align*} 
    there exist $0<\varepsilon<1$ and  $C>0$ such that if $F \in C^N([0,1])$ then
    \begin{align}\label{hor6}
       \Big \Vert \big(\mathscr{P}^k F\big)(\cdot) - \tilde{u}_1(\cdot) \int_0^1 F \d x  - \beta^{-k}\, \tilde{u}_3(\cdot)\, \frac{F(1)-F(0)}{4} \,  \Big\Vert_{L^\infty([0,1])} \leq C \,\Vert F\Vert_{C^N([0,1])} \beta^{-k(1+\varepsilon)},
    \end{align}
for $k$ sufficiently large.
\end{theorem}
\begin{remark}
When $a_0=a_1$, it is proved in \cite{CHM} (even for the more general case with $a_0=\ldots =a_{n-1}$, $n\geq 2$) that the spectrum of $\mathscr{P}$ seen as an operator in $L^p([0,1])$, $1\leq p\leq 2$, equals the closed unit complex disk. Moreover, the whole open unit disk consists of eigenvalues, hence the eigenvalue $1$ is far from being isolated.  
\end{remark}

\begin{remark}
    Let $F \in C^\infty([0,1])$ be given by $F(x)=\psi_1(x)=\chi_{[0,1]}(x)$. We see from the first two equations of \eqref{hc20} that $F=\tilde{u}_1+\tilde{u}_2$, and from the first equation of \eqref{hor5} we have: 
    \begin{align}\label{eq:optimalex1}
        \mathscr{P}^k\, F = \tilde{u}_1+\bigg(-\frac{a_1}{\beta^2}\bigg)^k \, \tilde{u}_2, \quad \forall k\geq 1.
    \end{align}
    Using $\int_0^1 F(x)dx=1$ and $F(0)=F(1)$ in \eqref{hor6} we see that in general we cannot hope to have a third term of order $\beta^{-2k}$ in the expansion because 
    \begin{align*}
        \beta^{2k} \Big (\big(\mathscr{P}^k F\big)(y) - \tilde{u}_1(y)\Big ) = (-a_1)^k \tilde{u}_2(y)
    \end{align*}
    where the above right hand side does not converge when $k \to \infty$. We conjecture that if $F\in C^\infty([0,1])$, then  \eqref{hor6} only holds if 
    \[0<\epsilon <1-\frac{\ln(a_1)}{\ln(\beta)}.\]
    The $\epsilon$ we provide in our proof (see \eqref{epsil}) is far from the conjectured optimal value. 
\end{remark}

\begin{remark}
A related question is the following: given $a_0$ and $a_1$ (thus also $\beta$, see \eqref{eq:a0betaa0plus1}), what is the minimal value of $N$ for which the convergence in \eqref{hor6} holds? If $a_1=1$, our lower bound is $N>6$, but we conjecture that $N>2$ should be enough in this case. 
\end{remark}

\begin{remark}
    The case with $n> 2$ and with general coefficients $a_0,\ldots, a_{n-1}$ satisfying Parry's conditions \cite{Pa} remains open, but we believe that if $\beta$ is a Pisot number 
    (i.e. a positive algebraic integer greater than $1$, whose conjugate elements have absolute value less than $1$), then (a less explicit) two-term asymptotics would still be possible. The case in which infinitely many coefficients $a_k$ are non-zero is completely open. 
\end{remark}

\section{Proof of Theorem \ref{thm1}}
We will show in \ref{sec:4dinvariantsubspace} that the subspace ${\rm Span}\{\psi_j:\, 1\leq j\leq 4\}$ is left invariant by $\mathscr{P}$, while the $\tilde{u}_j$'s with $1\leq j\leq 3$ are some carefully chosen eigenfunctions corresponding to three of its eigenvalues: $1$, $-a_1/\beta^2$ and $\beta^{-1}$.  The value of the integrals in \eqref{hor5} can be checked using \eqref{hc20} and the fact that $\int_0^1 \psi_1 \d x=\int_0^1 \psi_2\d x=1$ and $\int_0^1 \psi_3 \d x=\int_0^1 \psi_4\d x=0$. 

From now on we will focus on the proof of \eqref{hor6}. The very first step consists in partitioning the interval $[0,1]$ in a clever way, so that any smooth function $F$ defined on it can be approximated by an ``Euler-Bernoulli expansion" (see \eqref{hor13}) where the building blocks are functions which have a ``good behavior'' when $\mathscr{P}$ is repeatedly applied to them (see Lemmas \ref{lemmacrux} and \ref{lemmaPk}).  

We start by splitting the interval $[0,1]$ into two intervals  $[0,a_0\beta^{-1}]$ and $[a_0\beta^{-1},1]$. Due to \eqref{eq:fundamentaleqmoregeneral}, the length of the latter interval equals $a_1\beta^{-2}$. These two intervals are then split into $a_0$ and $a_1$ intervals of length $\beta^{-1}$ and $\beta^{-2}$, respectively. The endpoints of these $a_0$ and $a_1$ intervals are given by (see Figure~\ref{fig:splitting} for an illustration of this splitting):
\begin{align*}
    t_{k_1}^{j_1} = \begin{cases}
        \frac{j_1}{\beta} & 0 \leq j_1 \leq a_0, \qquad  \textup{when } k_1=1 \\
        \frac{a_0}{\beta} + \frac{j_1}{\beta^2} & 0 \leq j_1 \leq a_1, \qquad  \textup{when } k_1=2
    \end{cases}.
\end{align*}
Note that this representation of points is not unique, for example $t_1^{a_0}=t_2^{0}$.

\begin{figure}[!h]
\begin{center}
\begin{tikzpicture}
  \draw[scale=10,-] (0, 0) -- (1, 0) node[right] {};
  \draw[-] (0, -0.1) -- (0, 0.1) node[below] {};
  \node[] at (0, -0.4)   (a) {$0$};
  \draw[-] (10, -0.1) -- (10, 0.1) node[below] {};
  \node[] at (10, -0.4)   (a) {$1$};
  \draw[-] (6.18, -0.1) -- (6.18, 0.1) node[below] {};
  \node[] at (6.18, -0.4)   (a) { $\frac{a_0}{\beta}$};

  \draw[-] (1, -0.1) -- (1, 0.1) node[below] {};
  \node[] at (1, -0.4)   (a) {$t_1^{1}$};
  \draw[-] (2, -0.1) -- (2, 0.1) node[below] {};
  \node[] at (2, -0.4)   (a) {$t_1^{2}$};
  \draw[-] (3, -0.1) -- (3, 0.1) node[below] {};
  \node[] at (3, -0.4)   (a) {$t_1^{3}$};

  \node[] at (4, -0.4)   (a) {$\cdots$};

  \draw[-] (5.18, -0.1) -- (5.18, 0.1) node[below] {};
  \node[] at (5.18, -0.4)   (a) {$t_1^{a_0-1}$};

  \draw[-] (6.93, -0.1) -- (6.93, 0.1) node[below] {};
  \node[] at (6.93, -0.4)   (a) {$t_2^{1}$};

   \draw[-] (7.68, -0.1) -- (7.68, 0.1) node[below] {};
  \node[] at (7.68, -0.4)   (a) {$t_2^{2}$};

    \node[] at (8.35, -0.4)   (a) {$\cdots$};

  \draw[-] (9.25, -0.1) -- (9.25, 0.1) node[below] {};
  \node[] at (9.25, -0.4)   (a) {$t_2^{a_1-1}$};

  \draw[black, decorate, decoration={brace, amplitude=5, raise=5}] (0, 0) -- (1, 0) node[pos=0.5, above=8] {$\frac{1}{\beta}$};

  \draw[black, decorate, decoration={brace, amplitude=5, raise=5}] (6.18, 0) -- (6.93, 0) node[pos=0.5, above=8] {$\frac{1}{\beta^2}$};
  
\end{tikzpicture}
\caption{Plot of the first splitting of the interval $[0,1]$.}
\label{fig:splitting}
\end{center}
\end{figure}
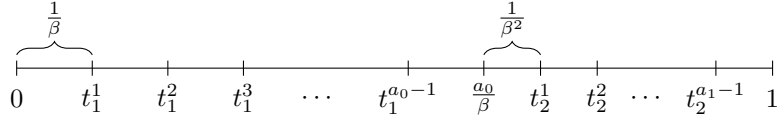

A further splitting gives smaller intervals with endpoints
\begin{align*}
    t_{k_1k_2}^{j_1j_2} = t_{k_1}^{j_1} + \beta^{-k_1} t_{k_2}^{j_2}
\end{align*}
and in general after $n$ splits
\begin{align*}
    t_{k_1k_2\cdots k_n}^{j_1j_2\cdots j_n} = t_{k_1k_2\cdots k_{n-1}}^{j_1j_2\cdots j_{n-1}} + \beta^{-k_1-k_2-\cdots - k_{n-1}} t_{k_n}^{j_n}.
\end{align*}
Let us introduce the notation 
\begin{equation}\label{hor7}
\begin{aligned}
\K_n&:=(k_1,k_2,\ldots,k_n),\quad \J_n:=(j_1,j_2,\ldots,j_n),\\
t_{\K_n}^{\J_n}&:=t_{k_1k_2\cdots k_n}^{j_1j_2\cdots j_n},\quad  t_{\K_n}^{\J_{n-1},j_n+1} := t_{k_1\cdots k_n}^{j_1\cdots j_{n-1} (j_n+1)},\\
 |\K_n|&:=k_1+k_2+\cdots +k_n.
\end{aligned}
\end{equation}
The points $t_{\K_n}^{\J_n}$ and $t_{\K_n}^{\J_{n-1},j_n+1}$ are ``consecutive'' points which obey 
\begin{equation}\label{hor9}
    t_{\K_n}^{\J_{n-1},j_n+1}-t_{\K_n}^{\J_{n}}=\beta^{-|\K_n|}. 
\end{equation}
By the rescaled Bernoulli approximation formula \eqref{BAF} with $b=t_{\K_n}^{\J_{n-1},j_n+1}$ and $a=t_{\K_n}^{\J_n}$ we have
\begin{equation}\label{hor8}
    \begin{aligned}
     F(y)&=  \beta^{|\K_n|} \int_{t_{\K_n}^{\J_n} }^{t_{\K_n}^{\J_{n-1},j_n+1} } F(u) \d u \\
     &\qquad +\sum_{s=1}^N \beta^{-(s-1)|\K_n|}\big ( F^{(s-1)}(t_{\K_n}^{\J_{n-1},j_n+1} )-F^{(s-1)}(t_{\K_n}^{\J_n} )\big ) \frac{\widetilde{B}_s\Big( \beta^{|\K_n|}\big (y-t_{\K_n}^{\J_n}\big )\Big)}{s!} \\
     &\qquad -\beta^{-N|\K_n|}\int_{0}^{1}\, F^{(N)}\big (t_{\K_n}^{\J_n} (1-u)+t_{\K_n}^{\J_{n-1},j_n+1} u\big )\, \frac{\widetilde{B}_N\Big( \beta^{|\K_n|}\big (y-t_{\K_n}^{\J_n}\big )-u\Big)}{N!} \d u \\
     &=   \beta^{|\K_n|} \int_{t_{\K_n}^{\J_n} }^{t_{\K_n}^{\J_{n-1},j_n+1} } F(u) \d u  \\
     &\qquad +\sum_{s=1}^N \beta^{-(s-1)|\K_n|}\big ( F^{(s-1)}(t_{\K_n}^{\J_{n-1},j_n+1} )-F^{(s-1)}(t_{\K_n}^{\J_n} )\big ) \frac{\widetilde{B}_s\Big( \beta^{|\K_n|}\big (y-t_{\K_n}^{\J_n}\big )\Big)}{s!}\\
     &\qquad +\mathcal{O}\Big( \beta^{-|\K_n|N}\Vert F^{(N)}\Vert_{L^\infty([0,1])}\Big),\quad \forall y\in \big [t_{\K_n}^{\J_n}\, ,\, t_{\K_n}^{\J_{n-1},j_n+1}\big ].
    \end{aligned}
    \end{equation} 
We now introduce a parameter $M\in \N$ which later on will be chosen to depend on $N$. Given any $M\in \N$, let us argue that we may find a partition of the interval $[0,1]$ such that the distance between any two consecutive points is either $\beta^{-M}$ or $\beta^{-(M+1)}$. 
\begin{enumerate}[label=(\roman*)]
    \item If $M=1$ then this already holds for the ``first layer'' of points $t_{k_1}^{j_1}$. 
    \item If $M=2$, then we leave the interval between $[a_0/\beta,1]$  untouched, while all the intervals $[t_1^{j_1},t_1^{j_1+1}]$ with $0\leq j_1\leq a_0-1$ have to be refined by going to the ``second layer'' of points of the form $t_{1 k_2}^{j_1j_2}$. The distance between two consecutive points of the second layer is either $\beta^{-2}$ when $k_2=1$ or $\beta^{-3}$ when $k_2=2$. 

\item If $M=3$ we need to further refine the intervals between points with $k_1=k_2=1$ (by going to the ``third layer'') and those with $k_1=2$ (by going to the ``second layer''). And so on. 
\end{enumerate}

By this refining procedure we end up with a partition of $[0,1]$ containing points from different layers such that any point in the partition obeys $|\K_n| \simeq M$ (meaning either equal to $M$ or $M+1$). 
\begin{definition}\label{Giz}
Let $B_s$ be the Bernoulli polynomials for $s \in \mathbb{N}_0$. Then we define
{\footnotesize    \begin{align*}
        G_{s,\K_n}^{\J_n} (x) \coloneqq \frac{\beta^{|\K_n|}}{\Vert B_s \Vert_{L^1([0,1])}}B_s \Big( \beta^{|\K_n|} \Big(x-t_{\K_n}^{\J_n}\Big)\Big) \, \chi_{\big[t_{\K_n}^{\J_n},t_{\K_n}^{\J_{n-1},j_n+1}\big]}(x).
    \end{align*}}
    and 
    \begin{equation}\label{hor10}
       \psi_{2s+1}(x):= \frac{B_{s}(x)}{\Vert B_{s} \Vert_{L^1([0,1])}} \chi_{[0,1]}(x),\quad \psi_{2(s+1)}(x) := 
    \frac{\beta}{a_1}\frac{B_s\big( \frac{\beta}{a_1}x \big)}{\big\Vert B_s  \big\Vert_{L^1([0,1])}} \chi_{[0,a_1\beta^{-1}]}(x),\quad  s\in \N_0. 
    \end{equation}
\end{definition}
\begin{remark}
    We note that when $s=0$ and $s=1$, the functions $\psi_1,\dots, \psi_4$ provided by \eqref{hor10} coincide with those introduced in \eqref{ks6}. 
\end{remark}

Using  \eqref{hor8} and Definition \ref{Giz} for  a partition of $[0,1]$ where the distance between any two consecutive points is $\beta^{-|\K_n|}\simeq\beta^{-M}$, we have in the sup-norm:
    \begin{equation}\label{hor13}
    \begin{aligned}
        F(y) &= \sum_{\K_n,\J_n} \bigg\{  G_{0,\K_n}^{\J_n} (y)\, \int_{t_{\K_n}^{\J_n}}^{t_{\K_n}^{\J_{n-1},j_n+1}} F(u) \d u   \\
        &\qquad + \sum_{s=1}^N \beta^{-|\K_n|s}\Big( F^{(s-1)}(t_{\K_n}^{\J_{n-1},j_n+1})-F^{(s-1)}(t_{\K_n}^{\J_{n}})\Big)  \frac{1}{s!}\Vert B_s \Vert_{L^1([0,1])} \, G_{s,\K_n }^{\J_n}(y)\bigg\} \\
        & \qquad +\mathcal{O}\Big(\beta^{-MN} \big \Vert F^{(N)}\Vert_{L^\infty([0,1])} \Big),\quad \forall y\in [0,1].
    \end{aligned}
    \end{equation}
    The following lemma plays a crucial role:
    \begin{lemma}\label{lemmacrux}
Let $s \in \mathbb{N}_0$. Then for almost every $x\in [0,1]$ and 
  {\bf independently of the choice of} $\J_n$ we have 
  \begin{align*}
        \Big(\mathscr{P}^{|\K_n|} G_{s,\K_n}^{\J_n}\Big)(x) = \psi_{2s+1}(x).
    \end{align*}
\end{lemma}
We refer the reader to \ref{sec:iterativeappandtransfop} for more details and the proof of this result. We will use it together with \eqref{hor13} and \eqref{hor2} in order to derive the following important lemma: 
\begin{lemma}\label{lemmaPk}
 In the $L^\infty$ norm, uniformly in $k>M+1$ we have: 
\begin{equation}\label{hor11}
\begin{aligned}
    &\big(\mathscr{P}^k F\big)(y) = \tilde{u}_1(y)\int_0^1 F(t) \d t \, + \tilde{u}_3(y)\, \frac{F(1)-F(0)}{4} \beta^{-k} \\
    &+ \mathcal{O}\bigg( \bigg(\frac{a_1}{\beta^2}\bigg)^{k-M} \Vert F\Vert_{C^N([0,1])}\bigg) + \mathcal{O}\big( \beta^{-k-M} \Vert F\Vert_{C^N([0,1])}\big) + \mathcal{O}\Big(\beta^{ k-MN} \big \Vert F\Vert_{C^N([0,1])} \Big).
\end{aligned} 
\end{equation}
\end{lemma}
  The proof of this lemma is postponed to \ref{sec:proofoflemmaasymptformulapk}. There are two main ideas behind this proof. The first one is that for every fixed $\nu\in \N$, the functions $\psi_j$ with $1\leq j\leq 2\nu$ form an invariant $2\nu$-dimensional subspace for $\mathscr{P}$, and one can analyze in quite some detail its restriction to this subspace (which will be done in \ref{sec:restrictionmatrixandproject}), i.e. its eigenvalues and corresponding eigenfunctions. The second main idea is that by applying sufficiently many times $\mathscr{P}$ to the ``building blocks'' $G_{s,\K_n}^{\J_n}$ from \eqref{hor13}, we eventually end up in an invariant subspace, as shown in Lemma \ref{lemmacrux}.

To finish the proof of Theorem~\ref{thm1}, we are only left with showing the existence of $\epsilon$ in \eqref{hor6}, which amounts to bound the three error terms appearing in \eqref{hor11}. 
Choosing $M \coloneqq \lfloor \frac{3k}{N}\rfloor$, then both $\beta^{-k-M}\sim \big (\beta^{-1-3/N})^k$ and
$\beta^{k-MN} \sim \beta^{-2 k}$
go faster to zero than $\beta^{-k}$ as $k \to \infty$. Furthermore, we have
\begin{align*}
    \bigg( \frac{a_1}{\beta^2}\bigg)^{k-M} \sim \beta^{-k} \Big (\Big( \frac{a_1}{\beta}\Big)^{1-3/N}\beta^{3/N} \Big )^k=\beta^{-k} \Big (\Big( \frac{a_1}{\beta}\Big)\Big( \frac{\beta^2}{a_1}\Big)^{3/N}\Big )^k,
\end{align*}
which goes to zero faster than $\beta^{-k}$ if $N$ is chosen as in the theorem. Finally, our choice for $\epsilon$ is:
\begin{equation}\label{epsil}
\epsilon:= \min \Big \{ \frac{3}{N}\, ,\, \frac{-(3/N)\ln (\beta^2/a_1)+\ln(\beta/a_1) }{\ln(\beta)} \Big \}.
\end{equation}
\qed

\appendix

\section{Various properties of the transfer operator}
\subsection{The Euler-Bernoulli approximation formula revisited}\label{sec:eulerbernoulliformula}

Here we will prove \eqref{june17}, see also Formula 23.1.32 in \cite{AS}. Our starting point is the inversion formula for Fourier series which under our regularity condition $f\in C^N([0,1])$ with $N\geq 1$ it holds for all $x\in (0,1)$: 
\[
f(x)=\int_0^1 f(t) \d t +\sum_{m\neq 0} e^{2\pi i mx}f_m,\quad f_m=\int_0^1 e^{-2\pi i mt} f(t) \d t.
\] 
By integration by parts for $m\neq 0$ we have: 
\begin{align*}
f_m &=(-2\pi i m)^{-1}\int_0^1 f(t) \Big (e^{-2\pi i mt}\Big )' \d t\\
&=(-2\pi i m)^{-1} \big (f(1)-f(0)\big )+(2\pi i m)^{-1}\int_0^1 f'(t)\,  e^{-2\pi i mt} \d t\\
&=-\sum_{s=1}^N (2\pi i m)^{-s}\big (f^{(s-1)}(1)-f^{(s-1)}(0)\big )+(2\pi i m)^{-N}\int_0^1 f^{(N)}(t) \, e^{-2\pi i mt} \d t.
\end{align*}
Introducing this expression back into the Fourier series and using the definition of the periodized Bernoulli polynomials we obtain \eqref{june17}. We note that the series $-\sum_{m\neq 0} (2\pi i m)^{-1} e^{2\pi i m x}$ converges uniformly to $B_1(x)=x-1/2$ on any interval $[\epsilon,1-\epsilon]$ with $0<\epsilon<1/2$. 

Another elementary proof, not using Fourier theory, is as follows: consider the  periodized Bernoulli polynomial $\widetilde{B}_1(x-u)$ with a fixed $x\in (0,1)$ given by
\[
\widetilde{B}_1(x-u)=\left \{\begin{matrix}
    x-u-1/2, & 0\leq u\leq x, \\
    x-u+1/2, & x<u\leq 1,
\end{matrix}\right . 
\]
hence 
\begin{align*}
\int_0^1 f'(u) \widetilde{B}_1(x-u) \d u &= \int_0^x f'(u) (x-u-1/2) \d u+\int_x^1 f'(u) (x-u+1/2) \d u\\
&=-f(x)+\int_0^1 f(u)\d u +\big (f(1)-f(0)\big ) \widetilde{B}_1(x).
\end{align*}
Now we use the recursion formula $\partial_u \widetilde{B}_n(x-u)=-n \widetilde{B}_{n-1}(x-u)$ for the Bernoulli polynomials if $u\neq x$, together with integration by parts and the fact that $\widetilde{B}_n$ is everywhere continuous if $n\geq 2$. 

We can also rescale \eqref{june17} as follows.  
Assume that $F \colon [a,b]\to \R$ belongs to  $C^N([a,b])$ and define $f(x)=F\big (a(1-x)+bx\big )$ with $0\leq x\leq 1$. Then for all $a<y<b$ we have 
\begin{equation}\label{BAF}
    \begin{aligned}
     F(y)&=   (b-a)^{-1}\int_a^b F(u) \d u +\sum_{s=1}^N (b-a)^{s-1}\big ( F^{(s-1)}(b)-F^{(s-1)}(a)\big ) \frac{\widetilde{B}_s\big( \frac{y-a}{b-a}\big)}{s!} \\
     &\qquad -(b-a)^N\int_{0}^{1}\, F^{(N)}\big (a(1-u)+bu\big )\, \frac{\widetilde{B}_N\big (\frac{y-a}{b-a}-u \big )}{N!} \d u. 
    \end{aligned}
\end{equation}
\subsection{The Bernoulli polynomials are eigenfunctions in the integer case} \label{lem:lemA2}
First, we immediately have that $\T \widetilde{B}_0=\widetilde{B}_0$ and also: 
\begin{align*}
\big (\T \widetilde{B}_1\big )(x)=x/q+q^{-1}(q-1)/2 -1/2=q^{-1}(x-1/2)=q^{-1}\widetilde{B}_1(x),\quad 0\leq x<1.
\end{align*}

Second, if $n\geq 2$ we have:
\begin{equation}\label{hc1}
\begin{aligned}
    (\T \widetilde{B}_n)(x)&=-n! \sum_{m\neq 0}  (2\pi i m)^{-n} \frac{1}{q}\sum_{j=0}^{q-1}e^{2\pi i m \big (\frac{x+j}{q}\big )}
    \\&=-n!\sum_{m\neq 0}  (2\pi i m)^{-n}e^{\frac{2\pi i m  x}{q}}\frac{1}{q} \sum_{j=0}^{q-1}\big (e^{\frac{2\pi i m }{q}}\big )^j\\
    &=-n!\sum_{k\neq 0}  (2\pi i kq)^{-n} e^{2\pi i k x} \\
    &=q^{-n}\widetilde{B}_n(x),
\end{aligned}
\end{equation}
where we used that if $m/q\not\in \mathbb{Z}$ then $\sum_{j=0}^{q-1}\big (e^{\frac{2\pi i m }{q}}\big )^j=0$.

\subsection{The expression of the transfer operator when the base is not an integer}
\begin{lemma}\label{lemmaA1}
    If $\beta$ satisfies \eqref{eq:fundamentaleqmoregeneral}, then the transfer operator is given by \eqref{ks5}.
\end{lemma}
\begin{proof}
    Using $(K_\beta g)(x)=g(\beta x-\lfloor \beta x\rfloor)$ with $g\in L^\infty([0,1])$ we have for every $f\in L^1([0,1])$ that
    \begin{align*}
     \int_0^1 f(x) (K_\beta g)(x)\d x &=\sum_{j=0}^{a_0-1}\int_{\frac{j}{\beta}}^{\frac{j+1}{\beta}}f(x) g(\beta x-j)\d x 
        +\int_{\frac{a_0}{\beta}}^1 f(x) g(\beta x -a_0)\d x\\
        &=\beta^{-1}\sum_{j=0}^{a_0-1}\int_{0}^1 f\Big (\frac{x+j}{\beta}\Big ) g(x)\d x +\beta^{-1}\int_{0}^{\frac{a_1}{\beta}} f\Big (\frac{x+a_0}{\beta}\Big ) g(x)\d x \\
        & =\int_0^1 (\mathscr{P} f)(x) g(x)\d x, 
    \end{align*}
    
    where in the second equality we have used that $\beta-a_0=a_1\beta^{-1}$ and also, that the last integral in the same equality can be extended to the whole interval $[0,1]$ because $f\big (\frac{x+a_0}{\beta}\big )=0$ when $\frac{a_1}{\beta}<x$.
\end{proof}
The following lemma gives a few additional properties of the transfer operator $\mathscr{P}$.
\begin{lemma}\label{lem:propertiesofmathscrP}
Let  $f \in L^\infty([0,1])$. Then it follows that
    \begin{enumerate}
    \item $\vert (\mathscr{P}f)(x)\vert \leq (\mathscr{P}\vert f \vert)(x)$.
        \item  $\Vert \mathscr{P}f\Vert_{L^\infty([0,1])} \leq \frac{a_0+1}{\beta}\Vert f \Vert_{L^\infty([0,1])}<\beta \, \Vert f \Vert_{L^\infty([0,1])}$.
        \item $\int_0^1 (\mathscr{P}f)(x) \d x = \int_0^1 f(x) \d x$.
    \end{enumerate}
\end{lemma}
\begin{proof} Straightforward computations, except maybe for the inequality $a_0+1<\beta^2$, which follows from: 
\[ \beta^2=a_0\beta +a_1>a_0+a_1\geq a_0+1.\]
\end{proof}

\section{Proof of Lemma~\ref{lemmacrux}}\label{sec:iterativeappandtransfop}
%
%
%
We recall the splitting into smaller intervals of the interval $[0,1]$ illustrated in Figure~\ref{fig:splitting}. With the simplified notation from \eqref{hor7}
we see that the general formula for the splitting can be written as
\begin{align*}
    t_{\K_n}^{\J_n} = t_{\K_{n-1}}^{\J_{n-1}} + \beta^{-\vert \K_{n-1} \vert} t_{k_n}^{j_n}.
\end{align*}
We have the following useful result:
\begin{lemma}\label{lem:splittingofpoints0to1}
Let $n \in \mathbb{N}$, then 
    \begin{align*}
        t_{1k_2\cdots k_n}^{\J_n} &= \frac{j_1}{\beta} + \beta^{-1} t_{k_2\cdots k_n }^{j_2 \cdots j_n}, \\
        t_{2k_2\cdots k_n}^{\J_n} &= \frac{a_0}{\beta} + \beta^{-1} t_{1k_2\cdots k_n }^{\J_n}
    \end{align*}
\end{lemma}
\begin{proof}
We have
\begin{align*}
    t_{1k_2\cdots k_{n}}^{\J_n} &= t_{1k_2\cdots k_{n-1}}^{\J_{n-1}}+\beta^{-1-k_2-\cdots - k_{n-1}}t_{k_{n}}^{j_{n}} \\
        &= j_1\beta^{-1} + \beta^{-1}t_{k_2}^{j_2} + \beta^{-1 - k_2} t_{k_3}^{j_3} +\cdots + \beta^{-1 - k_2 - \cdots - k_{n-1}}t_{k_{n}}^{j_{n}} \\
        &= j_1\beta^{-1} + \beta^{-1} t_{k_2\cdots k_{n}}^{j_2 \cdots j_n}
        \end{align*}
        and 
        \begin{align*}
    t_{2k_2\cdots k_n}^{\J_n} &= t_{2k_2\cdots k_{n-1}}^{\J_{n-1}}+\beta^{-2-k_2-\cdots - k_{n-1}}t_{k_{n}}^{j_{n}} \\
        &= a_0\beta^{-1}+ j_1\beta^{-2} + \beta^{-2}t_{k_2}^{j_2} + \beta^{-2 - k_2} t_{k_3}^{j_3} +\cdots + \beta^{-2 - k_2 - \cdots - k_{n-1}}t_{k_{n}}^{j_{n}} \\
        &= a_0\beta^{-1} + \beta^{-1} t_{1k_2\cdots k_{n}}^{\J_n}.
        \end{align*}
\end{proof}
Note that by a change of variable $u = \beta^{\vert \K_n\vert }\Big( x - t_{\K_n}^{\J_n}\Big)$, it follows that the functions introduced in Definition \ref{Giz} satisfy:
\begin{align*}
        \Big\Vert G_{s,\K_n}^{\J_n} \Big\Vert_{L^1([0,1])} = 1.
    \end{align*}
We now have all the ingredients we need to prove Lemma~\ref{lemmacrux}, which will be done by induction in $n$. For the basis step, let $n=1$. First, let $k_1=1$, then for $0 \leq j_1 \leq a_0-1$:
\begin{align*}
    G_{s,1}^{j_1}(x) &= \frac{\beta}{\Vert B_s \Vert_{L^1([0,1])}} B_s \Big(\beta\big(x - t_1^{j_1}\big)\Big) \chi_{\big[t_1^{j_1},t_1^{j_1+1} \big]}(x) \\
    &= \frac{\beta}{\Vert B_s \Vert_{L^1([0,1])}} B_s \bigg(\beta\bigg(x - \frac{j_1}{\beta}\bigg)\bigg) \chi_{\big[j_1\beta^{-1},(j_1+1)\beta^{-1} \big]}(x).
\end{align*}
 Thus by applying the operator $\mathscr{P}$, we get for a.e. $x \in [0,1]$:
\begin{align*}
    \Big(\mathscr{P}G_{s,1}^{j_1}\Big)(x) &= \frac{1}{\beta}\sum_{j=0}^{a_0} G_{s,1}^{j_1}\bigg( \frac{x+j}{\beta}\bigg)  \\
    &= \frac{1}{\beta}\sum_{j=0}^{a_0} \frac{\beta}{\Vert B_s \Vert_{L^1([0,1])}} B_s \bigg(\beta\bigg(\frac{x+j}{\beta} - \frac{j_1}{\beta}\bigg)\bigg) \chi_{\big[j_1\beta^{-1},(j_1+1)\beta^{-1} \big]}\bigg( \frac{x+j}{\beta}\bigg) \\
    &= \sum_{j=0}^{a_0} \frac{1}{\Vert B_s \Vert_{L^1([0,1])}} B_s \big(x+j -j_1\big) \chi_{[j_1-j,j_1-j+1 ]}(x)\chi_{[0,1]}(x) \, .
\end{align*}
Since $0 \leq j_1 \leq a_0-1$, the only term in the sum which is different from zero a.e. is the one with $j=j_1$, hence
    \begin{align}
        \Big(\mathscr{P}G_{s,1}^{j_1}\Big)(x) 
        &= \frac{B_s(x)}{\Vert B_s\Vert_{L^1([0,1])}}\chi_{[0,1]}(x)=\psi_{2s+1}(x). \label{eq:mathscrPGnequal1}
    \end{align}
    Next, let $k_1=2$. Then for $0\leq j_1 \leq a_1-1$: 
    \begin{align*}
        G^{j_1}_{s,2} (x)&= \frac{\beta^2}{\Vert B_s \Vert_{L^1([0,1])}}B_s\Big( \beta^2 \big( x - t_2^{j_1} \big) \Big)\chi_{\big[t_2^{j_1},t_2^{j_1+1} \big]}(x) \\
        &=\frac{\beta^2}{\Vert B_s \Vert_{L^1([0,1])}}B_s\bigg( \beta^2 \bigg( x - \frac{a_0}{\beta} - \frac{j_1}{\beta^2} \bigg) \bigg)\chi_{[a_0\beta^{-1} + j_1\beta^{-2},a_0\beta^{-1}+(j_1+1)\beta^{-2}]}(x)
    \end{align*}
and 
\begin{align*}
    & \Big(\mathscr{P}G_{s,2}^{j_1}\Big)(x) = \frac{1}{\beta}\sum_{j=0}^{a_0} G_{s,2}^{j_1}\bigg( \frac{x+j}{\beta}\bigg) \\
    &= \frac{1}{\beta}\sum_{j=0}^{a_0} \frac{\beta^2}{\Vert B_s \Vert_{L^1([0,1])}} B_s \bigg(\beta^2\bigg(\frac{x+j}{\beta} - \frac{a_0}{\beta} -  \frac{j_1}{\beta^2}\bigg)\bigg) \chi_{[a_0\beta^{-1} + j_1\beta^{-2},a_0\beta^{-1}+(j_1+1)\beta^{-2}]}\bigg( \frac{x+j}{\beta}\bigg) \\
    &= \sum_{j=0}^{a_0} \frac{\beta}{\Vert B_s \Vert_{L^1([0,1])}} B_s \bigg(\beta\bigg(x+j- a_0 -  \frac{j_1}{\beta} \bigg)\bigg) \chi_{[a_0 + j_1\beta^{-1}-j,a_0+(j_1+1)\beta^{-1}-j]}(x)\chi_{[0,1]}(x) \, .
\end{align*}
Only the term with $j=a_0$ is not zero almost-everywhere and thus we get
\begin{align*}
    \Big(\mathscr{P}G_{s,2}^{j_1}\Big)(x) 
    &= \frac{\beta}{\Vert B_s \Vert_{L^1([0,1])}} B_s \Big( \beta \big(x - t_1^{j_1}\big)\Big) \chi_{\big[ t_1^{j_1},t_1^{j_1+1}\big]}(x) = G_{s,1}^{j_1}(x).
\end{align*}
Another application of $\mathscr{P}$ and using \eqref{eq:mathscrPGnequal1} give
\begin{align*}
    \Big(\mathscr{P}^2G_{s,2}^{j_1}\Big)(x) = \Big(\mathscr{P}G_{s,1}^{j_1}\Big)(x) = \psi_{2s+1}(x),
\end{align*}
which finishes the basis step of our induction proof. For the induction step, assume that the equality is true for $n=m$ and let $n=m+1$. 

We first consider the case when $k_1=1$. By Lemma~\ref{lem:splittingofpoints0to1} we may write
\begin{align*}
        &G_{s,1k_2\cdots k_{m+1}}^{\J_{m+1}} (x) = \frac{\beta^{1+k_2+\cdots + k_{m+1}}}{\Vert B_s \Vert_{L^1([0,1])}}B_s \Big( \beta^{1+k_2+\cdots+k_{m+1}} \Big(x-t_{1k_2\cdots k_{m+1}}^{\J_{m+1}}\Big)\Big) \chi_{\big[t_{1k_2\cdots k_{m+1}}^{\J_{m+1}},t_{1k_2 \cdots k_{m+1}}^{\J_m,j_{m+1}+1}\big]}(x) \\
        &= \frac{\beta^{1+k_2+\cdots+ k_{m+1}}}{\Vert B_s \Vert_{L^1([0,1])}}B_s \bigg( \beta^{1+k_2+\cdots+k_{m+1}} \bigg(x-\frac{j_1}{\beta }-\beta^{-1}t_{k_2\cdots k_{m+1}}^{j_2\cdots j_{m+1}}\bigg)\bigg) \\
        &\qquad \cdot \chi_{\big[j_1\beta^{-1} + \beta^{-1}t_{k_2\cdots k_{m+1}}^{j_2 \cdots j_{m+1}},j_1\beta^{-1}+\beta^{-1}t_{k_2 \cdots k_{m+1}}^{j_2\ldots j_{m}(j_{m+1}+1)}\big]}(x),
    \end{align*}
    for $0 \leq j_1 \leq a_0-1$. Thus for a.e. $x \in [0,1]$ we have 
    \begin{align*}
        &\Big( \mathscr{P}^{1 + k_2 + \cdots + k_{m+1}} G_{s,1k_2\cdots k_{m+1}}^{\J_{m+1}} \Big)(x) = \Big( \mathscr{P}^{k_2 + \cdots + k_{m+1}} \Big( \mathscr{P} G_{s,1k_2\cdots k_{m+1}}^{\J_{m+1}} \Big)\Big)(x) \\
        &= \mathscr{P}^{k_2+\cdots +k_{m+1}}\bigg(\frac{1}{\beta}\sum_{j=0}^{a_0} G_{s,1k_2\cdots k_{m+1}}^{\J_{m+1}}\bigg( \frac{x+j}{\beta}\bigg)\chi_{[0,1]}(x) \bigg) \\
        &= \Big(\mathscr{P}^{k_2+\cdots +k_{m+1}}G_{s,k_2\cdots k_{m+1}}^{j_2\cdots j_{m+1}}\Big)(x) \\
        &=\psi_{2s+1}(x),
    \end{align*}
    where in the third equality we have used that the only term in the sum which is not zero a.e. has $j=j_1$. Furthermore, in the last equality we have used the induction hypothesis.

    Next let $k_1=2$. By Lemma~\ref{lem:splittingofpoints0to1} we may write
    \begin{align*}
        &G_{s,2k_2\cdots k_{m+1}}^{\J_{m+1}} (x) = \frac{\beta^{2+k_2+\cdots + k_{m+1}}}{\Vert B_s \Vert_{L^1([0,1])}}B_s \Big( \beta^{2+k_2+\cdots+k_{m+1}} \Big(x-t_{2k_2\cdots k_{m+1}}^{\J_{m+1}}\Big)\Big) \chi_{\big[t_{2k_2\cdots k_{m+1}}^{\J_{m+1}},t_{2k_2 \cdots k_{m+1}}^{\J_m,j_{m+1}+1}\big]}(x) \\
        &= \frac{\beta^{2+k_2+\cdots + k_{m+1}}}{\Vert B_s \Vert_{L^1([0,1])}}B_s \bigg( \beta^{2+k_2+\cdots+k_{m+1}} \bigg(x-\frac{a_0}{\beta}-\beta^{-1}t_{1k_2\cdots k_{m+1}}^{\J_{m+1}}\bigg)\bigg) \\
        &\qquad \cdot \chi_{\big[a_0\beta^{-1} + \beta^{-1}t_{1k_2\cdots k_{m+1}}^{\J_{m+1}},a_0\beta^{-1}+\beta^{-1}t_{1k_2 \cdots k_{m+1}}^{\J_m,j_{m+1}+1}\big]}(x),
        \end{align*}
        for $0 \leq j_1 \leq a_1-1$. Thus for a.e. $x \in [0,1]$ we have
        \begin{align*}
        &\Big( \mathscr{P}^{2 + k_2 + \cdots + k_{m+1}} G_{s,2k_2\cdots k_{m+1}}^{\J_{m+1}} \Big)(x) = \Big( \mathscr{P}^{1+k_2 + \cdots + k_{m+1}} \Big( \mathscr{P} G_{s,2k_2\cdots k_{m+1}}^{\J_{m+1}} \Big)\Big)(x) \\
        &= \mathscr{P}^{1+k_2+\cdots +k_{m+1}}\bigg(\frac{1}{\beta}\sum_{j=0}^{a_0} G_{s,2k_2\cdots k_{m+1}}^{\J_{m+1}}\bigg( \frac{x+j}{\beta}\bigg)\chi_{[0,1]}(x) \bigg )\\
        &= \Big(\mathscr{P}^{1+k_2+\cdots +k_{m+1}}G_{s,1k_2\cdots k_{m+1}}^{\J_{m+1}}\Big)(x) \\
        &=\psi_{2s+1}(x),
    \end{align*}
    where in the third equality, we again have used that the only non-zero term a.e. has $j=a_0$, and in the fourth equality we have used the already proved case with $k_1=1$.
\qed
\section{Invariant subspaces and their restriction matrices}\label{sec:restrictionmatrixandproject}

\subsection{The $4$-dimensional invariant subspace}\label{sec:4dinvariantsubspace}
Recall the functions $\psi_1$, $\psi_2$, $\psi_3$ and $\psi_4$ defined in \eqref{ks6}. 
By applying the operator $\mathscr{P}$, we get for a.e. $x \in[0,1]$:
\begin{align*}
    (\mathscr{P}\psi_1)(x) 
    &= \frac{1}{\beta}\sum_{j=0}^{a_0} \chi_{[0,1]}\bigg(\frac{x+j}{\beta}\bigg) \\
    &= \frac{1}{\beta}\sum_{j=0}^{a_0} \chi_{[-j,\beta-j]}(x)\chi_{[0,1]}(x) \\
    &= \frac{1}{\beta}\sum_{j=0}^{a_0-1} \chi_{[0,1]}(x) + \frac{1}{\beta} \chi_{[0,a_1\beta^{-1}]}(x) \\
    &= \frac{a_0}{\beta}\psi_1(x) + \frac{a_1}{\beta^2}\psi_2(x),
\end{align*}
where in the third equality we have used that $-j\leq 0< 1< \beta-j$ when $0\leq j \leq a_0 -1$ due to \eqref{eq:a0betaa0plus1}, and that \eqref{eq:fundamentaleqmoregeneral} implies $\beta-a_0 = a_1\beta^{-1}$.

By a similar computation we get:
\begin{align*}
    (\mathscr{P}\psi_2)(x) 
    &= \psi_1(x).
\end{align*}
Also:
 \begin{align*}
     &(\mathscr{P}\psi_3)(x) \\ 
    &= \frac{4}{\beta} \sum_{j=0}^{a_0} \bigg(\frac{x+j}{\beta}-\frac{1}{2}\bigg) \chi_{[0,1]}\bigg(\frac{x+j}{\beta}\bigg)  \\
    &= \frac{4}{\beta} \sum_{j=0}^{a_0} \bigg(\frac{x+j}{\beta}-\frac{1}{2}\bigg) \chi_{[-j,\beta-j]}(x)\chi_{[0,1]}(x) \\
     &= \frac{4}{\beta} \sum_{j=0}^{a_0-1} \bigg(\frac{x+j}{\beta}-\frac{1}{2}\bigg)\chi_{[0,1]}(x) + \frac{4}{\beta} \bigg(\frac{x+a_0}{\beta} - \frac{1}{2}\bigg)\chi_{[0,a_1\beta^{-1}]}(x) \\
      &=\frac{4}{\beta^2}\bigg( a_0x + \frac{(a_0-1)a_0}{2} -\frac{a_0}{2}\bigg(a_0 + \frac{a_1}{\beta}\bigg)\bigg)\chi_{[0,1]}(x) + \frac{4}{\beta^2} \bigg(x+ a_0 - \frac{a_0}{2} -\frac{a_1}{2\beta}\bigg)\chi_{[0,a_1\beta^{-1}]}(x) \\
      &= -\frac{2a_0a_1}{\beta^3}\psi_1(x) + \frac{2 a_1 a_0}{\beta^3} \psi_2(x) + \frac{a_0}{\beta^2}\psi_3(x) + \frac{a_1^2}{\beta^4} \psi_4(x),
 \end{align*}
 where in the third equality we have used \eqref{eq:a0betaa0plus1}, and in the fourth we have used \eqref{eq:fundamentaleqmoregeneral} twice. 
Similarly, we have
\begin{align*}
    &(\mathscr{P}\psi_4)(x) \\
    &= \frac{1}{\beta}\sum_{j=0}^{a_0} \psi_4\bigg(\frac{x+j}{\beta}\bigg)  \\
    &=\frac{4\beta}{a_1^2} \sum_{j=0}^{a_0-1} \bigg(\frac{x+j}{\beta}- \frac{a_1}{2\beta}\bigg) \chi_{[-j,a_1-j]}(x)\chi_{[0,1]}(x) + \frac{4}{a_1^2} \bigg(x+a_0- \frac{a_1}{2}\bigg) \chi_{[-a_0,a_1-a_0]}(x) \chi_{[0,1]}(x)\\
    &= \frac{4\beta}{a_1^2} \sum_{j=0}^{a_1-1} \bigg(\frac{x+j}{\beta}- \frac{a_1}{2\beta}\bigg) \chi_{[0,1]}(x) + 0 \\
    &= \frac{1}{a_1}\psi_3(x).
\end{align*}
%
Thus the subspace generated by $\psi_1,\psi_2,\psi_3,\psi_4$ is left invariant by $\mathscr{P}$, and the corresponding restriction matrix is given by:
\begin{align*}
    \mathcal{P}_{4} = \begin{bmatrix}
        a_0\beta^{-1} & 1 & -2a_0a_1\beta^{-3} & 0\\
        a_1\beta^{-2} & 0 & 2a_0a_1\beta^{-3}  & 0\\
        0 & 0 & a_0\beta^{-2} & a_1^{-1}\\
        0 & 0 & a_1^2\beta^{-4} & 0
    \end{bmatrix}.
\end{align*}
Note that $\mathcal{P}_4$ is of the form
\begin{align}\label{july13}
    \mathcal{P}_4 = \begin{bmatrix}
        A & B \\
        {\bf 0}_2 & C
    \end{bmatrix},
\end{align}
with ${\bf 0}_2,A,B,C \in \mathbb{R}^{2 \times 2}$. We have
\begin{align*}
    A = \begin{bmatrix}
        a_0 \beta^{-1} & 1 \\
        a_1 \beta^{-2} & 0
    \end{bmatrix},
\end{align*}
and its two eigenvalues $\lambda_1$ and $\lambda_2$ must satisfy that
\begin{align*}
    \mathrm{Tr}(A) &= \frac{a_0}{\beta} = \lambda_1 + \lambda_2, \quad 
    \mathrm{Det}(A) = -\frac{a_1}{\beta^2} = \lambda_1 \lambda_2.
\end{align*}
One can check that $\lambda_1 = 1$ and $\lambda_2 = -a_1\beta^{-2}$ satisfy these conditions.

Similarly, we have
\begin{align*}
    C= \begin{bmatrix}
        a_0\beta^{-2} & a_1^{-1} \\
        a_1^2 \beta^{-4} & 0
    \end{bmatrix}
    \end{align*}
    and its two eigenvalues $\lambda_3$ and $\lambda_4$ satisfy that
\begin{align*}
    \mathrm{Tr}(C) &= \frac{a_0}{\beta^2} = \lambda_3 + \lambda_4 ,\quad 
    \mathrm{Det}(C) = -\frac{a_1}{\beta^4} = \lambda_3 \lambda_4.
\end{align*}
One can again check that $\lambda_3 = \beta^{-1}$ and $\lambda_4 = -a_1\beta^{-3}$ satisfy these conditions.

By a direct computation it follows that for $z \neq \lambda_j, \; j=1,2,3,4$
\begin{align}\label{eq:mathcalPresolventmatrix}
(z\bb1_4 -\mathcal{P}_4)^{-1} =\begin{bmatrix}
        (z\bb1_2 -A)^{-1}& (z\bb1_2 -A)^{-1} \, B \, (z\bb1_2 -C)^{-1}\\
        {\bf 0}_2& (z\bb1_2 -C)^{-1}
    \end{bmatrix}.    
\end{align}
We are interested in computing the Riesz projections corresponding to the eigenvalues $\lambda_1$, $\lambda_2$ and $\lambda_3$ of the matrix $\mathcal{P}_4$, which we denote by $\Pi_1$, $\Pi_2$ and $\Pi_3$, respectively.

For $j \in \{1,2\}$, let $\Gamma_j$ be a positively oriented circle centered at $\lambda_j$ which does not contain any other eigenvalue. Then the projection corresponding to $\lambda_j$ is
\begin{align*}
    \Pi_j=\frac{1}{2\pi i} \int_{\Gamma_j} (z\bb1_4 -\mathcal{P}_4)^{-1} \d z.
\end{align*}
By an elementary computation we get
\begin{align*}
    (z \bb1_2 - A)^{-1} &= \frac{1}{(z-1)(z+a_1\beta^{-2})}\begin{bmatrix}
        z & 1 \\
        a_1\beta^{-2} & z- a_0\beta^{-1}
    \end{bmatrix},
\end{align*}
which by residue calculus implies that
\begin{align*}
     \pi_1 :=\frac{1}{2\pi i} \int_{\Gamma_1} (z\bb1_2 -A)^{-1} \d z &= \frac{1}{\beta^2+a_1}\begin{bmatrix}
        \beta^2 & \beta^2 \\
        a_1 & a_1
    \end{bmatrix} \\
     \pi_2 :=\frac{1}{2\pi i} \int_{\Gamma_2} (z\bb1_2 -A)^{-1} \d z &= -\frac{1}{\beta^2+a_1}\begin{bmatrix}
        -a_1 & \phantom{-}\beta^2 \\
        \phantom{-}a_1 & -\beta^2
    \end{bmatrix}
\end{align*}
and using that $\pi_1 B=\mathbf{0_2}$ (see \eqref{july13}) we have
\begin{align}\label{pi1}
    \Pi_1 = \begin{bmatrix}
        \frac{\beta^2}{\beta^2+a_1} & \frac{\beta^2}{\beta^2+a_1} & 0 &0 \\
        \frac{a_1}{\beta^2+a_1} & \frac{a_1}{\beta^2+a_1} & 0 & 0 \\
        0 & 0 & 0 & 0 \\
        0 & 0 & 0 & 0
    \end{bmatrix}, \qquad 
    \Pi_2 = \begin{bmatrix}
        \frac{a_1}{\beta^2+a_1} & \frac{-\beta^2}{\beta^2+a_1} & \star &\star \\
        \frac{-a_1}{\beta^2+a_1} & \frac{\beta^2}{\beta^2+a_1} & \star & \star \\
        0 & 0 & 0 &0\\
        0 & 0 &0& 0
    \end{bmatrix},
\end{align}
where the stars ``$\star$'' stand for $\pi_2 B(\lambda_2 \bb1_2 -C)^{-1}$ which we do not need to explicitly compute. Note that the first column of $\Pi_1$ and $\Pi_2$ provide the coefficients of $\tilde{u}_1$ and $\tilde{u}_2$ of their expansion in \eqref{hc20}.

Similarly, let $\Gamma_3$ be a positively oriented circle centered at $\lambda_3=\beta^{-1}$, not containing the other eigenvalues. Then the projection corresponding to $\lambda_3$ is 
\begin{align*}
\Pi_3=\frac{1}{2\pi i} \int_{\Gamma_3} (z\bb1_4 -\mathcal{P}_4)^{-1} \d z.
\end{align*}
By another straightforward calculation one gets 
\begin{align*}
    (z \bb1_2 - C)^{-1} &= \frac{1}{(z-\beta^{-1})(z+a_1\beta^{-3})}\begin{bmatrix}
        z & a_1^{-1} \\
        a_1^2\beta^{-4} & z- a_0\beta^{-2}
    \end{bmatrix},
\end{align*}
and hence
\begin{align*}
    \frac{1}{2\pi i} \int_{\Gamma_3} (z\bb1_2 -A)^{-1} \d z = \mathbf{0}_2, \qquad     \pi_3\coloneqq\frac{1}{2\pi i}  \int_{\Gamma_3} (z\bb1_2 -C)^{-1} \d z = \frac{1}{\beta^2 + a_1}\begin{bmatrix}
        \beta^2 & a_1^{-1}\beta^3 \\
        a_1^2\beta^{-1} & a_1
    \end{bmatrix},
\end{align*}
where we have used that $(z\bb1_2-A)^{-1}$ is analytic in a neighborhood of $\lambda_3 = \beta^{-1}$. Notice that we also have
\begin{align*}
    (z \bb1_2-C)^{-1} = \frac{1}{z - \beta^{-1}}\pi_3 + \textup{ holomorphic terms.}
\end{align*}
By \eqref{eq:mathcalPresolventmatrix}, \eqref{july13} and the above calculations, the projection matrix $\Pi_3$ is then
\begin{align*}
    \Pi_3 = \frac{1}{2\pi i} \int_{\Gamma_3} (z\bb1_4 -\mathcal{P}_4)^{-1} \d z = \begin{bmatrix}
        {\bf 0}_2 & (\beta^{-1}\bb1_2 - A)^{-1}B \pi_3 \\
        {\bf 0}_2 & \pi_3
    \end{bmatrix}.
\end{align*}
By yet another calculation, we get
\begin{align*}
     (\beta^{-1}\bb1_2 - A)^{-1}B\pi_3 = \frac{2a_0a_1}{(\beta + a_1)(\beta^2+a_1)}\begin{bmatrix}
        -\beta & -a_1^{-1}\beta^2 \\
    \phantom{-}\beta & \phantom{-} a_1^{-1}\beta^2 
    \end{bmatrix} 
\end{align*}
and finally that 
\begin{align}\label{pi3}
    \Pi_3 = \begin{bmatrix}
        0 & 0 & \frac{-2a_0a_1\beta}{(\beta+a_1)(\beta^2+a_1)} & \frac{-2a_0\beta^{2}}{(\beta+a_1)(\beta^2+a_1)} \\
        0 & 0 & \frac{2a_0a_1\beta}{(\beta+a_1)(\beta^2+a_1)} & \frac{2a_0\beta^{2}}{(\beta+a_1)(\beta^2+a_1)} \\
         0 & 0 & \frac{\beta^2}{(\beta^2+a_1)} & \frac{a_1^{-1}\beta^3}{\beta^2+a_1} \\
          0 & 0 &  \frac{a_1^2\beta^{-1}}{(\beta^2+a_1)} & \frac{a_1}{\beta^2 + a_1} \\
    \end{bmatrix}.
\end{align}
Note that the third column of $\Pi_3$ provides the coefficients in the expansion of $\tilde{u}_3$ in \eqref{hc20}.
\subsection{The $2\nu$-dimensional invariant subspace}
Let $\nu \in \mathbb{N}$, $j \in \{0,1,\ldots,\nu-1\}$ and let us consider the functions $\psi_{2j+1}$ and $\psi_{2(j+1)}$ from \eqref{hor10}, which span a subspace of dimension $2\nu$. We are interested in computing the action of $\mathscr{P}$ on the term containing the highest power $x^j$ in both $\psi_{2j+1}$ and $\psi_{2(j+1)}$. Let $P_{j-1}(x,k)$ denote a generic polynomial in $x$ of degree $j-1$. Then:
\begin{align*}
    &\mathscr{P}\bigg( \frac{x^j \chi_{[0,1]}(x)}{\Vert B_j \Vert_{L^1([0,1])}}\bigg) \\
    &= \frac{\beta^{-1}}{ \Vert B_j \Vert_{L^1([0,1])}} \bigg(\sum_{k=0}^{a_0-1} \bigg( \frac{x^j}{\beta^j} + P_{j-1}(x,k) \bigg) \chi_{[0,1]}(x) + \bigg( \frac{x^j}{\beta^j} + P_{j-1}(x,a_0)\bigg)\chi_{[0,a_1\beta^{-1}]}(x) \bigg) \\
    &= \frac{\beta^{-1}}{ \Vert B_j \Vert_{L^1([0,1])}} a_0 \frac{x^j}{\beta^j} \chi_{[0,1]}(x) + \frac{\beta^{-1}}{ \Vert B_j \Vert_{L^1([0,1])}} \frac{x^j}{\beta^j} \chi_{[0,a_1\beta^{-1}]}(x) + R_{j-1}(x)\\
    &= \frac{a_0}{\beta^{j+1}}\psi_{2j+1}(x) + \frac{a_1^{j+1}}{\beta^{2j+2}}\psi_{2(j+1)}(x) + \tilde{R}_{j-1}(x),
\end{align*}
where in the first equality we have used that $a_0 < \beta < a_0+1$ together with the fundamental equation \eqref{eq:fundamentaleqmoregeneral}, while $R_{j-1}$ and $\tilde{R}_{j-1}$ are remaining terms of lower order, which we are not interested in. Thus we have proved that:
\[
\mathscr{P}\psi_{2j+1}=\sum_{s=0}^{2j} C_s \,\psi_s +\frac{a_0}{\beta^{j+1}}\psi_{2j+1} + \frac{a_1^{j+1}}{\beta^{2j+2}}\psi_{2(j+1)}\, .
\]
Now let us write an expansion for $\mathscr{P}\psi_{2(j+1)}$. 
Since the inequality $0\leq (x+a_0)\beta^{-1}\leq a_1\beta^{-1}$ is equivalent with $-a_0\leq x\leq a_1-a_0\leq 0$, we have that for almost all $x$:
\begin{equation*}
\chi_{[0,a_1 \beta^{-1}]}\big ((x+a_0)\beta^{-1}\big )\, \chi_{[0,1]}(x)=0.
\end{equation*}
This leads to
\begin{align*}
    \mathscr{P}\bigg(  \frac{\beta^{j+1} x^j \chi_{[0,a_1\beta^{-1}]}(x)}{a_1^{j+1}\Vert B_j \Vert_{L^1([0,1])}}\bigg) 
    &=\frac{\beta^{j}}{ a_1^{j+1}\Vert B_j \Vert_{L^1([0,1])}} \bigg(\sum_{k=0}^{a_0-1} \bigg( \frac{x+k}{\beta}\bigg)^j \chi_{[-k,a_1-k]}(x) \chi_{[0,1]}(x) + 0 \bigg) \\
    &= \frac{\beta^{j}}{ a_1^{j+1}\Vert B_j \Vert_{L^1([0,1])}} \bigg(\sum_{k=0}^{a_1-1} \bigg(\frac{x^j}{\beta^j} + P_{j-1}(x,k)\bigg)\chi_{[0,1]}(x) \bigg) \\
    &=\frac{\beta^{j}}{ a_1^{j+1}\Vert B_j \Vert_{L^1([0,1])}} a_1\frac{x^j}{\beta^j}  \chi_{[0,1]}(x) + R_{j-1}(x) \\
    &= \frac{1}{a_1^j} \psi_{2j+1}(x) + \tilde{R}_{j-1}(x),
\end{align*}
where $P_{j-1}$ is again a polynomial of degree $j-1$ in $x$ and $R_{j-1}$ and $\tilde{R}_{j-1}$ contain terms of lower order. We have thus proved that 
\[
\mathscr{P}\psi_{2(j+1)}=\sum_{s=0}^{2j} C_s \,\psi_s +\frac{1}{a_1^j} \psi_{2j+1}\, .
\]

Therefore, the general restriction matrix $\mathcal{P}_{2\nu}\in \R^{2\nu\times 2\nu}$ has the form
\begin{align*}
\mathcal{P}_{2\nu} = \begin{bmatrix}
        A_1 & \star & \star & \star & \star \\
        {\bf 0}_2 & \ddots & \star & \star & \star \\
        {\bf 0}_2 & {\bf 0}_2 &  A_j & \star & \star  \\
        {\bf 0}_2 & {\bf 0}_2 & {\bf 0}_2 &\ddots & \star \\
        {\bf 0}_2 & {\bf 0}_2 & {\bf 0}_2 & {\bf 0}_2 & A_{\nu}
    \end{bmatrix},
\end{align*}
where the stars denote some $2\times 2$ entries which we do not need to explicitly compute and $A_1,\ldots,A_{\nu} \in \mathbb{R}^{2\times 2}$ are given by 
\[
A_k=\begin{bmatrix}
    a_0\beta^{-k} & a_1^{-(k-1)}\\
    a_1^k \beta^{-2k} & 0
\end{bmatrix},\quad 1\leq k\leq \nu. 
\]

The spectrum of $\mathcal{P}_{2\nu}$ equals the union of spectra of $A_1,\ldots,A_{\nu}$. The spectrum of $A_k$  consists of two simple eigenvalues given by 
\begin{align*}
    \lambda_{2k-1}=\beta^{-k+1},\, \lambda_{2k}=-a_1\beta^{-k-1},\quad  1\leq k\leq \nu,
\end{align*}
with $\vert \lambda_i \vert \leq \beta^{-2}$ for every $i \geq 4$. We denote by $\{e_j\}_{1\leq j\leq 2\nu}$ the elements of the standard basis in $\C^{2\nu}$. Let us denote by $\Pi_j\in \C^{2\nu\times 2\nu}$ the rank-one projection corresponding to the eigenvalue $\lambda_j$. 

If $\Gamma_j$ is a positively oriented circle centered at $\lambda_j$ and not containing other eigenvalues we denote 
$$\pi_{j,k}:=\frac{1}{2\pi i} \int_{\Gamma_j} (z\bb1_2 -A_k)^{-1} \d z ,\quad 1\leq j\leq 2\nu,\quad 1\leq k\leq \nu.$$
This matrix equals $\mathbf{0}_2$  if $j$ is neither $2k-1$ nor $2k$, otherwise $\pi_{j,k}$ equals a rank-one orthogonal projection which can be explicitly computed and the first element of its first column is always different from zero. 
One can prove as in the case with $\nu=2$ that
\begin{align*}
\Pi_{j} = \begin{bmatrix}
        {\bf 0}_2 & \star & \star &  \star \\
        {\bf 0}_2 & \ddots & \star &  \star \\
        {\bf 0}_2 & {\bf 0}_2 &  \hspace{-0.5cm} \pi_{j,k}  & \star  \\
        {\bf 0}_2 & {\bf 0}_2 &\,\,  \quad  \ddots  &\star\\
        {\bf 0}_2 & {\bf 0}_2 & {\bf 0}_2  & {\bf 0}_2
    \end{bmatrix},\quad j\in \{2k-1,2k\},\quad 1\leq k\leq \nu. 
\end{align*}
This implies that the column number $2k-1$ of both $\Pi_{2k-1}$ and $\Pi_{2k}$ contains at least one non-zero element for all $1\leq k\leq \nu$.  Thus the vectors \[u_{1}:=\Pi_1 e_1,\, u_2:=\Pi_2 e_1, \, u_3:=\Pi_3 e_3,\, u_4:=\Pi_4 e_3,\, \dots , \, u_{2\nu-1}:=\Pi_{2\nu-1} e_{2\nu-1},\, u_{2\nu}:=\Pi_{2\nu} e_{2\nu-1}\]
 are all nonzero and are eigenvectors of $\mathcal{P}_{2\nu}$. 
 
 Since all eigenvalues are simple, these eigenvectors form a basis in $\C^{2\nu}$ and we may write 
\begin{equation*}
e_m=\sum_{i=1}^{2\nu} C_i^{(m)} u_i,\quad 1\leq m\leq 2\nu. 
\end{equation*}
Using the properties $\Pi_i\Pi_j=\Pi_i^2 \delta_{ij}$ and $\Pi_j u_i =u_j \delta_{ij}$, by applying $\Pi_{2k-1}$ to both sides of the above identity we have that for every $1\leq k\leq \nu$: 
\begin{equation}\label{hor3}
u_{2k-1}=\Pi_{2k-1} e_{2k-1}= C_{2k-1}^{(2k-1)}u_{2k-1}\Longrightarrow  C_{2k-1}^{(2k-1)}=1.
\end{equation}
%
%
%
%
Let $\mathcal{P}_{2\nu}=(p_{kj})_{1\leq k,j\leq 2\nu}$. Then $\tilde{u}_i(x) := \sum_{j=1}^{2\nu} (u_i)_j \psi_j(x)$, for $i \in \{1,2,\ldots,2\nu\}$, satisfies
\begin{align*}
    (\mathscr{P}\tilde{u}_i)(x) 
    = \sum_{j=1}^{2\nu} (u_i)_j \sum_{k=1}^{2\nu} p_{kj}\psi_k(x) 
    = \sum_{k=1}^{2\nu} \big(\mathcal{P}_{2\nu} u_i \big)_k\psi_k(x) 
    =\lambda_i \sum_{k=1}^{2\nu}  (u_i)_k\psi_k(x) 
    = \lambda_i \tilde{u}_i(x)
\end{align*}
and thus $\tilde{u}_i$ is an eigenfunction of $\mathscr{P}$. These eigenfunctions span the invariant subspace generated by the $\psi_m$'s.  Thus for all $1 \leq m \leq 2\nu$, we have
\begin{align*}
    \psi_m = \sum_{i=1}^{2\nu} C_i^{(m)} \tilde{u}_i,\quad C_{2k-1}^{(2k-1)}=1,\quad 1\leq k\leq \nu,
\end{align*}
 where in the second equality we used  \eqref{hor3}. 
Let $r \in \mathbb{N}$. By repeatedly applying $\mathscr{P}$ to $\psi_m$, we get
\begin{align}\label{eq:mathscrPrpsim}
    \mathscr{P}^r\psi_m &= \sum_{i=1}^{2\nu} C_i^{(m)} \lambda_i^r \tilde{u}_i = C_1^{(m)}\tilde{u}_1 + C_2^{(m)}\bigg( \frac{-a_1}{\beta^2} \bigg)^r\tilde{u}_2 + C_3^{(m)}\beta^{-r}\tilde{u}_3 + \mathcal{O}\big( \beta^{-2r}\big).
\end{align}

\begin{lemma}\label{lem:C1m}
    We have
    \begin{align*}
        C_1^{(m)} = \int_0^1 \psi_m(x) \d x
    \end{align*}
    which leads to
    \begin{align*}
        C_1^{(1)} = C_1^{(2)} = 1, \qquad \textup{ and } \qquad C_1^{(m)}=0, \qquad m \geq 3. 
    \end{align*}
    Moreover, 
    \begin{align}\label{hor1}
    \mathscr{P}^r\psi_3 &=  C_2^{(3)}\bigg( \frac{-a_1}{\beta^2} \bigg)^r\tilde{u}_2 + \beta^{-r}\tilde{u}_3 + \mathcal{O}\big( \beta^{-2r}\big)=\beta^{-r}\tilde{u}_3+\mathcal{O}\bigg(\bigg( \frac{a_1}{\beta^2} \bigg)^r\bigg ).
\end{align}
\end{lemma}
\begin{proof}
    Let $r \in \mathbb{N}$. From Lemma~\ref{lem:propertiesofmathscrP}$(iii)$ we have that for all $r\geq 1$:
    \begin{align*}
        \int_0^1 (\mathscr{P}^r\psi_m)(x) \d x = \int_0^1 \psi_m(x) \d x,
    \end{align*}
    which implies
\begin{align*}
    \int_0^1 \psi_m(x) \d x = \int_0^1 \bigg( C_1^{(m)}\tilde{u}_1(x) + C_2^{(m)}\bigg( \frac{-a_1}{\beta^2} \bigg)^r\tilde{u}_2(x) + C_3^{(m)}\beta^{-r}\tilde{u}_3(x) + \mathcal{O}\big( \beta^{-2r}\big) \bigg) \d x.
\end{align*}
By taking the limit $r \to \infty$, we get
\begin{align*}
    \int_0^1 \psi_m(x) \d x = \int_0^1 C_1^{(m)}\tilde{u}_1(x) \d x = C_1^{(m)},
\end{align*}
since $\int_0^1\tilde{u}_1(x)dx=1$. Then we note from \eqref{hor10} that 
\begin{align*}
    \int_0^1 \psi_m \d x = \begin{cases}
        1 & \textup{if } m \in \{1,2\}, \\
        0 & \textup{if } m \geq 3.
    \end{cases}
\end{align*}
Finally, \eqref{hor1} follows from \eqref{eq:mathscrPrpsim} in which we put $m=3$ and we use that $C_3^{(3)}=1$ (see \eqref{hor3}) and $C_1^{(3)}=0$.  
\end{proof}
Combining \eqref{eq:mathscrPrpsim} with Lemma~\ref{lem:C1m} gives for $r \in \mathbb{N}$, that
\begin{equation}\label{hor2}
\begin{aligned}
    (\mathscr{P}^r\psi_1)(x)  &= \tilde{u}_1(x) + \mathcal{O}\big( a_1^r\beta^{-2r} \big),  \\
    (\mathscr{P}^r\psi_3)(x)  &= \beta^{-r}\tilde{u}_3(x) + \mathcal{O}\big( a_1^r\beta^{-2r} \big),  \\
    (\mathscr{P}^r\psi_m)(x)  &= \mathcal{O}\big( \beta^{-r} \big), \quad \forall\, m\geq 4.
\end{aligned}
\end{equation}
\section{Proof of Lemma~\ref{lemmaPk}}\label{sec:proofoflemmaasymptformulapk}
We work with partitions of $[0,1]$ where the distance between any two consecutive points is $\beta^{-|\K_n|}\simeq\beta^{-M}$ and \eqref{hor13} holds. When $k > M+1$, the linearity of $\mathscr{P}$ and Lemma \ref{lemmacrux} imply that in $L^\infty$ sense 
\begin{align}\label{hc22}
   & \big(\mathscr{P}^k F\big)(y) = \sum_{\K_n,\J_n} \bigg\{  \mathscr{P}^{k-|\K_n|}\big(\psi_1(y)\big) \int_{t_{\K_n}^{\J_n}}^{t_{\K_n}^{\J_{n-1},j_n+1}} F(u) \d u  \\\nonumber 
        &+ \sum_{s=1}^N \beta^{-|\K_n|s}\Big( F^{(s-1)}\big(t_{\K_n}^{\J_{n-1},j_n+1}\big)-F^{(s-1)}\big(t_{\K_n}^{\J_{n}}\big)\Big)  \frac{1}{s!}\Vert B_s \Vert_{L^1([0,1])} \mathscr{P}^{k-|\K_n|} \big(\psi_{2s+1}(y)\big)\bigg\} \\ \nonumber 
        & +\mathcal{O}\Big(\beta^{ k-MN} \big \Vert F^{(N)}\Vert_{L^\infty([0,1])} \Big)
\end{align}
where in the last error term we used Lemma \ref{lem:propertiesofmathscrP}(ii). 
First, by \eqref{hor2} we have that the first term in \eqref{hc22} satisfies
\begin{align*}
  \sum_{\K_n,\J_n}  \mathscr{P}^{k-|\K_n|}\big(\psi_1(y)\big)\, \int_{t_{\K_n}^{\J_n}}^{t_{\K_n}^{\J_{n-1},j_n+1}} F(u) \d u &=\tilde{u}_1(y) \, \int_0^1 F(u) \d u + \mathcal{O}\bigg( \bigg( \frac{a_1}{\beta^2}\bigg)^{k-M}\Vert F\Vert_{L^\infty([0,1])}\bigg).  
\end{align*}

Secondly, using \eqref{hor2} again, the term with $s=1$ in \eqref{hc22}  satisfies (here $\Vert B_1 \Vert_{L^1([0,1])}=1/4$)
\begin{align*}
     &\sum_{\K_n,\J_n}\beta^{-|\K_n|}\frac{F\big(t_{\K_n}^{\J_{n-1},j_n+1}\big)-F(t_{\K_n}^{\J_{n}})}{4} \mathscr{P}^{k-|\K_n|} \big(\psi_3(y)\big)\\
     & =\sum_{\K_n,\J_n}\beta^{-|\K_n|}\frac{F\big(t_{\K_n}^{\J_{n-1},j_n+1}\big)-F\big(t_{\K_n}^{\J_{n}}\big)}{4} \bigg( \beta^{-k+|\K_n|}\tilde{u}_3(y) + \mathcal{O}\bigg( \bigg(\frac{a_1}{\beta^2}\bigg)^{k-M} \bigg)\bigg)\\
     & =\frac{F(1)-F(0)}{4}\beta^{-k}\tilde{u}_3(y) + \mathcal{O}\bigg(  \beta^{-M}\bigg(\frac{a_1}{\beta^2}\bigg)^{k-M} \Vert F'\Vert_{L^\infty} \bigg).
\end{align*}
Lastly, by using \eqref{hor2}, the remaining terms in the sum, i.e. $s \geq 2$, satisfy
\begin{align*}
    &\sum_{\K_n,\J_n}\sum_{s=2}^N \beta^{-s|\K_n|}\Big (F^{(s-1)}\big(t_{\K_n}^{\J_{n-1},j_n+1}\big)-F^{(s-1)}\big(t_{\K_n}^{\J_{n}}\big)  \Big ) \frac{1}{s!}\Vert B_s \Vert_{L^1([0,1])} \mathscr{P}^{k-|\K_n|} \big(\psi_{2s+1}(y)\big)\\
    &\qquad = \sum_{s=2}^N\beta^{-sM}\Vert F^{(s)}\Vert_{L^\infty} \, \mathcal{O}\big( \beta^{-k+M} \big) = \mathcal{O}\big( \beta^{-k-M} \Vert F\Vert_{C^N([0,1])}\big).
\end{align*}
Adding all these contributions together leads to \eqref{hor11}. \qed

\vspace{0.3cm}

{\bf Acknowledgments.} This work was supported by the Independent Research Fund Denmark–
Natural Sciences, grant DFF–10.46540/2032-00005B. We also thank the referees for their  valuable comments and suggestions which led to a significant improvement of the presentation.

\end{document}